\documentclass[
final,
a4paper
]{amsart}
\usepackage{amssymb,latexsym}

\usepackage[english]{babel}
\usepackage[utf8]{inputenc}

\usepackage[T1]{fontenc}
\usepackage{ae,aecompl}

\usepackage{enumitem}
\usepackage[notref,notcite,color]{showkeys}
\usepackage{verbatim}
\usepackage{url}

\usepackage{xspace,ifthen,dsfont}

\DeclareMathOperator{\lth}{lth}
\newcommand*{\lp}[2][]{\ensuremath{
    \mathrm{
      \ifthenelse{\equal{#1}{}}{#2}{#1 \text{-} #2}}}\xspace}
\newcommand*{\lpp}[3][]{\ensuremath{
    \mathrm{
      \ifthenelse{\equal{#1}{}}{#2}{#1 \text{-} #2}}(#3)}\xspace}
\newcommand*{\ls}[2][]{\ensuremath{
  \mathrm{
      \ifthenelse{\equal{#1}{}}{#2}{#1 \text{-} #2}}}\xspace}

\newcommand{\IMPL}[1][]{
  \ifthenelse{\equal{#1}{}}{\mathop{\rightarrow}}{\mathop{\stackrel{#1}{\longrightarrow}}}}

\newcommand*{\AND}{\mathrel{\land}}
\newcommand*{\OR}{\mathrel{\lor}}

\newcommand*{\IFF}[1][]{
  \ifthenelse{\equal{#1}{}}{\mathrel{\leftrightarrow}}{\mathrel{\stackrel{#1}{\longleftrightarrow}}}}

\newcommand*{\Quantor}[2]{{#1 #2}\,}
\newcommand*{\Forall}[1]{\Quantor{\forall}{#1}}
\newcommand*{\Exists}[1]{\Quantor{\exists}{#1}}

\newcommand*{\Nat}{\ensuremath{\mathds{N}}}
\newcommand*{\Real}{\ensuremath{\mathds{R}}}
\newcommand*{\Rat}{\ensuremath{\mathds{Q}}}

\newcommand*{\Theorem}{Theorem}
\newcommand*{\Proposition}{Proposition}
\newcommand*{\Lemma}{Lemma}
\newcommand*{\Corollary}{Corollary}
\newcommand*{\Definition}{Definition}
\newcommand*{\Remark}{Remark}
\newcommand*{\Notation}{Notation}

\theoremstyle{plain}
\newtheorem{theorem}{\Theorem}

\newtheorem{corollary}[theorem]{\Corollary}
\newtheorem{lemma}[theorem]{\Lemma}
\theoremstyle{definition}
\newtheorem{definition}[theorem]{\Definition}
\theoremstyle{remark}

\title{The cohesive principle and the Bolzano-Weierstra{\ss} principle}
\address{Fachbereich Mathematik, Technische Universit\"{a}t Darmstadt\\
Schlossgartenstra\ss e 7, 64289 Darmstadt, Germany}
\thanks{The author gratefully acknowledges the support by the German Science Foundation (DFG Project KO 1737/5-1).}
\thanks{I am grateful to Ulrich Kohlenbach for useful
discussions and suggestions for improving the presentation
of the material in this article.}
\author{Alexander P. Kreuzer}
\email{akreuzer@mathematik.tu-darmstadt.de}
\urladdr{http://www.mathematik.tu-darmstadt.de/~akreuzer}
\subjclass[2010]{03F60, 03D80, 03B30}
\keywords{Bolzano-Weierstra{\ss} principle, cohesive principle, sequential compactness}

\begin{document}
\begin{abstract}
  The aim of this paper is to determine the logical and computational strength of instances of the Bolzano-Weierstra\ss\ principle (\lp{BW}) and a weak variant of it.

  We show that \lp{BW} is instance-wise equivalent to the weak K\"onig's lemma for $\Sigma^0_1$\nobreakdash-trees (\lp[\Sigma^0_1]{WKL}). This means that from every bounded sequence of reals one can compute an infinite $\Sigma^0_1$\nobreakdash-0/1-tree, such that each infinite branch of it yields an accumulation point and vice versa. Especially, this shows that the degrees $d\gg 0'$ are exactly those containing an accumulation point for all bounded computable sequences.

   Let \lp{BW_{weak}} be the principle stating that every bounded sequence of real numbers contains a Cauchy subsequence (a sequence converging but not necessarily fast).
  We show that \lp{BW_{weak}} is instance-wise equivalent to the (strong) cohesive principle (\lp{StCOH}) and --- using this --- obtain a classification of the computational and logical strength of \lp{BW_{weak}}.
  Especially we show that \lp{BW_{weak}} does not solve the halting problem and does not lead to more than primitive recursive growth. Therefore it is strictly weaker than \lp{BW}.
 We also discuss possible uses of \lp{BW_{weak}}.
\end{abstract}
\maketitle

In this paper we investigate the logical and recursion theoretic strength of instances of the Bolzano-Weierstra\ss\ principle (\lp{BW}) and the weak variant of it stating only the existence of a slow converging Cauchy subsequence (\lp{BW_{weak}}). Slow converging means here that the rate of convergence does not need to be computable.

Let weak K\"onig's lemma (\lp{WKL}) be the principle stating that an infinite $0/1$-tree has an infinite branch and let \lp[\Sigma^0_1]{WKL} be the statement that an infinite $0/1$-tree given by a $\Sigma^0_1$-predicate has an infinite branch.

We show that \lp{BW} and \lp[\Sigma^0_1]{WKL} are \emph{instance-wise} equivalent. 
Instance-wise means here that for every instance of \lp{BW},
i.e.\ every bounded sequence, one can compute, uniformly, an instance of \lp[\Sigma^0_1]{WKL},
i.e.\ a code for an infinite $\Sigma^0_1$-0/1-tree,
such that from a solution of this instance of \lp[\Sigma^0_1]{WKL} one can compute, uniformly, an accumulation point and vice versa.
\emph{Instance-wise equivalence}  refines the usual logical equivalence where the full second order closure of the principles may be used --- e.g.\ arithmetical comprehension (\ls{ACA_0}, i.e.\ the schema $\Exists{X}\Forall{n} \left(n\in X \IFF \phi(n)\right)$ for any arithmetical formula $\phi$) and \lp[\Pi^0_1]{CA} (comprehension where $\phi$ is restricted to $\Pi^0_1$-formulas) are equivalent but they are not instance-wise equivalent. 
As consequence we obtain that the Turing degrees containing solutions to all instances of \lp[\Sigma^0_1]{WKL} (i.e.\ the degrees $d$ with $d\gg 0'$, see below) are exactly those containing an accumulation point for each computable bounded sequence. 

Furthermore, we show that \lp{BW_{weak}} is instance-wise equivalent to the 
strong cohesive principle, see Definition \ref{def:coh} below. Using this one can apply classification results obtained for the (strong) cohesive principle, see \cite{HS07,JS93,CJS01,CSYta}. Especially this shows that the $low_2$ degrees, i.e.\ degrees $d$ with $d'' \equiv 0''$,  are exactly those containing a slowly converging subsequence for every computable bounded sequence.
This shows also that \lp{BW_{weak}} does not lead to more than primitive recursive growth when added to \ls{RCA_0}.

\section{Cohesive Principle}

\begin{definition}\label{def:coh}
  Let $(R_n)_{n\in \Nat}$ be a sequence of subsets of $\Nat$.
  \begin{itemize}
  \item  A set $S$ is \emph{cohesive} for $(R_n)_{n\in \Nat}$ if
    $\Forall{n} \left(S\subseteq^* R_n \OR S \subseteq^* \overline{R_n}\right)$,\footnote{$A\subseteq^* B$ stands for $A \setminus B$ is finite.}
    i.e.
    \[
    \Forall{n} \Exists{s} \left(\Forall{j \ge s} \left(j\in S \IMPL j\in R_n\right) \OR \Forall{j\ge s} \left(j\in S \IMPL j\notin R_n\right) \right)
    .\]
  \item A set $S$ is \emph{strongly cohesive} for $(R_n)_{n\in \Nat}$ if 
    \[
    \Forall{n} \Exists{s} \Forall{i<n} \left(\Forall{j \ge s} \left(j\in S \IMPL j\in R_i\right) \OR \Forall{j\ge s} \left(j\in S \IMPL j\notin R_i\right) \right)
    .\]
  \item A set is called (\emph{p-cohesive}) \emph{r-cohesive} if it is cohesive for all (primitive) recursive sets.
  \end{itemize}
\end{definition}

\begin{definition}
  The \emph{cohesive principle} (\lp{COH}) is the statement that for every sequence of sets an infinite cohesive set exists.
  Similarly, the \emph{strong cohesive principle} (\lp{StCOH}) is the statement that for every sequence of sets an infinite strongly cohesive set exists.

  We will denote by \lpp{(St)COH}{X} the statement that for the sequence of sets $(R_n)_n$ coded by $X$  an infinite (strongly) cohesive set exists.
\end{definition}
Hirschfeldt and Shore showed in \cite[4.4]{HS07} that \lp{StCOH} is equivalent to $\lp{COH} \AND \lp[\Pi^0_1]{CP}$, where \lp[\Pi^0_1]{CP} is the $\Pi^0_1$-bounded collection princple 
\[
\Forall{n}\left(\Forall{x<n}\Exists{y} \phi(x,y) \IMPL \Exists{z}\Forall{x<n}\Exists{y<z}\, \phi(x,y)\right) \qquad\text{for any $\Pi^0_1$-formula $\phi$.}
\]
\lp[\Pi^0_1]{CP} follows from $\Sigma^0_2$-induction. Therefore there is no recursion theoretic difference between \lp{StCOH} and \lp{COH}.

The recursion theoretic strength of the cohesive principle is well understood, its reverse mathematical strength is a topic of active research mainly in the context of the classification of Ramsey's theorem for pairs, see \cite{HS07} for a survey.

To state the recursion theoretic strength of \lp{COH} we will need following notation. Denote by $a \gg b$ that the Turing degree $a$ contains an infinite computable branch for every $b$\nobreakdash-\hspace{0pt}computable 0/1-tree, see \cite{sS77}. In particular, the degrees $d\gg 0'$ are exactly those which contain an infinite path for every $\Sigma^0_1$-0/1-tree. By the low basis theorem for every $b$ there exists a degree $a \gg b$ which is $low$ over $b$, i.e.\ $a'\equiv b'$, see \cite{JS72}.

\begin{theorem}[\cite{JS93,JS93cor}, see also {\cite[theorem 12.4]{CJS01}}]\label{thm:js}
  For any degree $d$ the following are equivalent:
  \begin{itemize}
  \item There is an r-cohesive (p-cohesive) set with jump of degree $d$,
  \item $d \gg 0'$.
  \end{itemize}
  In particular, there exists a $low_2$ r-cohesive set.
\end{theorem}
\begin{theorem}\label{thm:cohcons}
  $\lp{COH}$ is $\Pi^1_1$-conservative over $\ls{RCA_0}$, $\ls{RCA_0}+\lp[\Pi^0_1]{CP}$, $\ls{RCA_0}+\lp[\Sigma^0_2]{IA}$.
\end{theorem}
This result for \ls{RCA_0} and $\ls{RCA_0}+\lp[\Sigma^0_2]{IA}$ is due to Cholak, Jockusch, Slaman, see \cite{CJS01}, the result for $\ls{RCA_0}+\lp[\Pi^0_1]{CP}$ is due to Chong, Slaman, Yang, see \cite{CSYta}.
\begin{corollary}
  $\ls{RCA_0}+\lp{StCOH}$ is $\Pi^0_2$-conservative over \ls{PRA}.
\end{corollary}
\begin{proof}
  Theorem~\ref{thm:cohcons} together with the fact that \lp[\Pi^0_1]{CP} is $\Pi^0_2$-conservative over \ls{PRA}.
\end{proof}

\section{Bolzano-Weierstra\ss\ principle}

Let \lp{BW} be the statement that every sequence $(y_i)_{i\in\Nat}$ of rational numbers in the interval $[0,1]$ admits a fast converging subsequence, that is a subsequence converging with the rate $2^{-n}$ or equivalently any other rate given by a computable function resp.\ by a function in the theory. This principle covers the full strength of Bolzano-Weierstra\ss, i.e.\ one can take a bounded sequence of real numbers.

Let \lp{BW_{weak}} be the statement that every sequence $(y_i)_{i\in\Nat}$ of rational numbers in the interval $[0,1]$ admits a Cauchy subsequence (a sequence converging but not necessarily fast), more precisely
\begin{multline*}
(\lp{BW_{weak}})\colon \\\Forall{(y_i)_{i\in\Nat}\subseteq \Rat \cap [0,1]} \Exists{f \text{ strictly monotone}} \Forall{n} \Exists{s} \Forall{v,w\ge s}\ |y_{f(v)} - y_{f(w)}| <_\Rat 2^{-n}
.\end{multline*}

The statement \lp{BW_{weak}} also implies that every bounded sequence of real numbers contains a Cauchy subsequence. Just continuously map the bounded sequence into $[0,1]$ and take a diagonal sequence of rational approximations of the elements of the original sequence.

We will denote by \lpp{BW}{Y}  and \lpp{BW_{weak}}{Y} the statement that the bounded sequence coded by $Y$ contains a (slowly) converging subsequence.

The principles \lp{BW} and \lp{BW_{weak}} also imply the corresponding Bolzano-Weierstra\ss\ principle for the Cantor space $2^\Nat$:
\begin{lemma}\label{lem:bwcant}
  Over \ls{RCA_0} 
  \begin{itemize}
  \item \lp{BW} implies the Bolzano-Weierstra\ss\ principle for the Cantor space $2^\Nat$ and
  \item \lp{BW_{weak}} implies the weak Bolzano-Weierstra\ss\ principle for the Cantor space $2^\Nat$, i.e.\ for every sequence in $2^\Nat$ there exists a slowly converging Cauchy subsequence.
  \end{itemize}

  Moreover these implications are \emph{instance-wise}, i.e.\ there exists an $e$ such that over \ls{RCA_0} the (weak) Bolzano-Weierstra\ss\ principles for a sequence $(x_i)_{i\in\Nat}\subseteq 2^\Nat$ coded by $X$ is implied by \lpp{BW_{(weak)}}{\{e\}^X}.

\end{lemma}
\begin{proof}
  Define the mapping $h\colon 2^\Nat \to [0,1]$ as
  \[
  h(x) = \sum_{i=0}^\infty \frac{2x(i)}{3^{i+1}}
  .\]
  The image of $h$ is the Cantor middle-third set.
  
  One easily establishes 
  \[
  dist_{2^{\Nat}}(x,y)<2^{-n} \quad\text{if{f}}\quad dist_{\Real}(h(x),h(y))<3^{-(n+1)}
  .\]
  Therefore (slow) Cauchy sequences of $2^\Nat$ primitive recursively correspond to (slow) Cauchy sequences of the Cantor middle-third set.

  For $\{e\}$ choose the function mapping $(x_i)_{i\in\Nat}$ to $(h(x_i))_{i\in\Nat}$.
  The lemma follows.
\end{proof}

The full Bolzano-Weierstra\ss\ principle (\lp{BW}) results from \lp{BW_{weak}}, if we additionally require an effective Cauchy-rate, e.g.\ $s=2^{-n}$ in the above definition of \lp{BW_{weak}}. One also obtains full \lp{BW} if one uses an instance of $\Pi^0_1$-comprehension (or Turing jump) to thin out the Cauchy sequence making it fast converging.

The weak version of the Bolzano-Weierstra\ss\ principle is for instance considered in computational analysis, see \cite[section 3]{RZ08}.

\lp{BW_{weak}} is also interesting in the context of proof-mining or ``hard analysis'', i.e.\ the extraction of quantitative information for analytic statements. For an introduction to hard analysis see \cite[\S 1.3]{tT08}, for proof-mining see \cite{uK08}. For instance if one uses \lp{BW_{weak}} to prove that a sequence converges, by theorem~\ref{thm:bwcons} below one can expect a primitive recursive rate of metastability, in the sense of Tao \cite[\S 1.3]{tT08}. Such proofs  occur in fixed-point theory, for example Ishikawa's fixed-point theorem uses such an argument, see \cite{uK05,sI76}.

Note that in this case only a single instance of the Bolzano-Weierstra\ss\ principle is used and the accumulation point is not used in  a $\Sigma^0_1$-induction, therefore one obtains the same results using Kohlenbach's elimination of Skolem functions for monotone formulas, see for instance \cite[theorem 1.2]{uK00}. Nested uses of \lp{BW} imply arithmetic comprehension and thus lead to non-primitive recursive growth.
In contrast to that, we will show that even nested uses of \lp{BW_{weak}} in a context with full $\Sigma^0_1$\nobreakdash-induction do not result in more than primitive recursive growth.

\section{Results}

\begin{theorem}\label{thm:bw}
  Over \ls{RCA_0} the principles \lp{BW} and \lp[\Sigma^0_1]{WKL} are instance-wise equivalent. More precisely 
  \begin{align*}
    \ls{RCA_0} &\vdash \Exists{e_1} \Forall{X} \left(\lpp[\Sigma^0_1]{WKL}{\{e_1\}^X} \IMPL \lpp{BW}{X}\right) ,\\
    \ls{RCA_0} &\vdash \Exists{e_2} \Forall{Y} \left(\lpp{BW}{\{e_2\}^Y} \IMPL \lpp[\Sigma^0_1]{WKL}{Y}\right)
  ,\end{align*}
  where \lpp[\Sigma^0_1]{WKL}{Y} is weak K\"onig's lemma for a $\Sigma^0_1$-tree coded by $Y$.

  In language with higher order functionals $\{e_1\}$ and $\{e_2\}$ could be given by fixed primitive recursive functionals.
\end{theorem}
\begin{proof}
  For the first implication see \cite{SK} and \cite[section 5.4]{uK98c}.

  For the converse implication note that \lp[\Sigma^0_1]{WKL} is instance-wise equivalent to $\Sigma^0_2$\nobreakdash-\hspace{0pt}separation, i.e.\ the statement that for two $\Sigma^0_2$-sets $A_0,A_1$ with $A_0\cap A_1 = \emptyset$ there exists a set $S$, such that $A_0 \subseteq S \subseteq \overline{A_1}$. This is for instance a consequence of \cite[lemma IV.4.4]{sS99} relativized to $\Delta^0_2$-sets. This proof of this lemma also yields a construction of the sets $A_0,A_1$, i.e.\ an $e'$ such that $\{e'\}^Y$ yields a set coding $A_0,A_1$.

  Thus is suffices to prove $\Sigma^0_2$-separation of two $\Sigma^0_2$-sets $A_0,A_1$.

  Let $B_i$ for $i<2$ be a quantifier free formula such that
  \[
  n \in \overline{A_i} \equiv \Forall{x}\Exists{y} B_i(x,y;n)
  .\]
  We assume that $y$ is unique; one can always achieve this by requiring $y$ to be minimal. 
  Note that by assumption $\Forall{x}\Exists{y} B_0(x,y;n) \OR \Forall{x}\Exists{y} B_1(x,y;n)$.

  Then define
  \[
  f_i(n,k) := \max \left\{ s < k \mid \Forall{x<\lth{s}} \left(B_i(x,(s)_x; n) \right)\right\}
  .\]
  We use here a sequence coding that is monotone in each component, i.e.\ for two sequences $s,t$ with the same length we have $s \le t$ if $(s)_x \le (t)_x$ for all $x<\lth(s)$, see for instance \cite[definition~3.30]{uK08}.

  If for fixed $n,i$ the statement $\Forall{x}\Exists{y} B_i(x,y;n)$ holds and $f_y$ is the choice function for $y$, i.e.\ the function satisfying $\Forall{x} B_i(x,f_y(x);n)$, then
  for the course-of-value function $\bar{f}_y$ of $f_y$
  \[
  f_i(n,\bar{f}_y(m)+1) = \bar{f}_y(m)
  .\]
  If $\Forall{x}\Exists{y} B_i(x,y;n)$ does not hold then $\lambda k .f_i(n,k)$ is bounded.
  Define $g_i(n,k):= \lth(f_i(n,k))$ and for each $n$ let $g_{i,n} := \lambda k. g_i(n,k)$.
  Then for each $i$
  \[
  \text{the range of } g_{i,n} \text{ is } \Nat \text{\quad if{f}\quad} \Forall{x}\Exists{y} B_i(x,y;n).
  \] 
  Therefore it is sufficient to find a set $S$ obeying
  \begin{equation}\label{eq:sreq}
  \Forall{n} \left( rng(g_{0,n}) \neq \Nat \IMPL n\in S \AND rng(g_{1,n}) \neq \Nat \IMPL n\notin S\right)
  .\end{equation}
  
  Define a sequence $(h_k)_{k\in\Nat} \subseteq 2^\Nat$ by
  \[
  h_k (n) :=
  \begin{cases}
    0 & \text{if } g_0(n,k) \ge g_1(n,k),\\
    1 & \text{otherwise.}
  \end{cases}
  \]
  By hypothesis, for each $n$ there is at least one $i<2$ such that the range of $g_{i,n}$ is $\Nat$.
  For a fixed $n$, if there is exactly one $i<2$, such that the range of $g_{i,n}$ is $\Nat$ then $\lim_{k\to\infty} h_k(n) = i$.
  In this case \eqref{eq:sreq} is satisfied for this $n$ if 
  \[
  n\in S \quad\text{if{f}}\quad \lim_{k\to\infty} h_k(n)=1
  .\]
  If for each $i<2$ the range $g_{i,n}$ is $\Nat$ then \eqref{eq:sreq} is trivially satisfied for this $n$.

  Applying \lp{BW} to $h_k$, yields an accumulation point $h$. For $h$ then
  \[
  h(n) = \lim_{k\to \infty} h_k(n) \quad\text{if the limit exists}
  .\]
  Hence  $h$ describes a characteristic function of a set $S$ obeying \eqref{eq:sreq}.

  A number  $e_2$ of a Turing machine such that $\{e_2\}^Y$ yields the Cantor middle-third set belonging to $(h_k)_k$ can easily be computed using $e$ from lemma~\ref{lem:bwcant} and $e'$.

  This proves the theorem.
\end{proof}

Since 
\[
\ls{RCA_0} \vdash \lp[\Sigma^0_1]{WKL}\IFF\lp[\Pi^0_1]{CA}
\]
 one obtains as consequence of this theorem that well known result that \lp{BW} is equivalent to \ls{ACA_0} over \ls{RCA_0}, see \cite[theorem I.9.1]{sS99}.

Notice that in Theorem~\ref{thm:bw} the use of \lp[\Sigma^0_1]{WKL} could neither be replaced by \lp[\Pi^0_1]{CA} nor \lp[\Pi^0_2]{CA}.

\begin{theorem}\label{thm:stcohbwweak}
  Over \ls{RCA_0} the principles \lp{BW_{weak}} and \lp{StCOH} are instance-wise equivalent. More precisely
  \begin{align*}
    \ls{RCA_0} &\vdash \Exists{e_1} \Forall{X} \left(\lpp{StCOH}{\{e_1\}^X} \IMPL \lpp{BW_{weak}}{X}\right) ,\\
    \ls{RCA_0} &\vdash \Exists{e_2} \Forall{Y} \left(\lpp{BW_{weak}}{\{e_2\}^Y} \IMPL \lpp{StCOH}{Y}\right).
  \end{align*}
  In a language with higher order functionals $\{e_1\}$ and $\{e_2\}$ could be given by fixed primitive recursive functionals.
\end{theorem}
\begin{proof}
  To prove \lp{BW_{weak}} for a sequence $(x_i)_{i\in\Nat}$ coded by $X$ define 
  \begin{align*}
    R_i&:=\left\{j \in \Nat \biggm| x_j \in\bigcup_{k \text{ even}}\left[\frac{k}{2^i},\frac{k+1}{2^i}\right]\right\}
    \intertext{and}
    R^y &:= \bigcap_{i<\lth(y)}
    \begin{cases}
      R_i & \text{if } (y)_i = 0, \\
      \overline{R_i} & \text{otherwise.}
    \end{cases}
  \end{align*}

  Let $f$ be a strictly increasing enumeration of a strongly cohesive set for $(R_i)_i$. Then by definition it follows, that
  \[
  \Forall{i}\Exists{y,s} \left(\lth(y)=i \AND \Forall{w>s}\ f(w)\in R^{y}\right)
  .\]
  This statement is equivalent to
  \[
  \Forall{i}\Exists{k,s} \Forall{w>s}\left( x_{f(w)}\in \left[\frac{k}{2^i},\frac{k+1}{2^{i}}\right]\right)
  ,\]
  which implies \lp{BW_{weak}}.
  Clearly there exists a number $e_1$ of a Turing machine computing $(R_i)_i$. The first part of the theorem follows.

  For the other direction, let $(R_i)_{i\in\Nat}$ be a sequence of sets coded by $Y$. Let $(x_i)_{i\in\Nat}\subseteq 2^\Nat$ be the sequence defined by
  \[
  x_i(n) :=
  \begin{cases}
    1 & \text{if } i\in R_n , \\
    0 & \text{if } i\notin R_n .
  \end{cases}
  \]
  Applying \lp{BW_{weak}} and lemma~\ref{lem:bwcant} to $(x_i)_i$ yields a slowly converging subsequence $(x_{f(i)})_{i\in\Nat}$, i.e.
  \[
  \Forall{n} \Exists{s} \Forall{j,j'\ge s} dist(x_{f(j)},x_{f(j')}) < 2^{-n}
  .\]
  By spelling out the definition of $dist$ and $x_i$ we obtain
  \[
  \Forall{n} \Exists{s} \Forall{j,j'\ge s} \Forall{i< n} \left(f(j) \in R_i \IFF f(j')\in R_i\right)
  ,\]
  which implies that the set strictly monotone enumerated by $f$ is strongly cohesive.

  The number $e_2$ can be easily computed using the construction in lemma~\ref{lem:bwcant}.
\end{proof}

As immediate corollary we obtain:
\begin{corollary}\label{cor:stcohbwweak}
  \[
  \ls{RCA_0} \vdash \lp{StCOH} \IFF \ls{BW_{weak}}
  \]
\end{corollary}
Hence all results for \lp{StCOH} carry over to \lp{BW_{weak}}:
\begin{theorem}\label{thm:bwcons}
  \lp{BW_{weak}} is $\Pi^1_1$-conservative over $\ls{RCA_0}+\lp[\Pi^0_1]{CP}$, $\ls{RCA_0}+\lp[\Sigma^0_2]{IA}$. Especially $\ls{RCA_0}+\lp{BW_{weak}}$  is $\Pi^0_2$-conservative over \ls{PRA}.
\end{theorem}
\begin{proof}
  Corollary~\ref{thm:stcohbwweak} and Theorem~\ref{thm:cohcons}.
\end{proof}

\begin{theorem}\mbox{}\label{thm:bwcomp}
  \begin{enumerate}
  \item Every recursive sequence of real numbers contains a $low_2$ Cauchy subsequence (a sequence converging but not necessarily fast).\label{enum:bwcomp1}
  \item There exists a recursive sequence of real numbers containing no computable Cauchy subsequence.\label{enum:bwcomp2}
  \item There exists a recursive sequence of real numbers containing no converging subsequence computable in $0'$.\label{enum:bwcomp3}
  \end{enumerate}
\end{theorem}
\begin{proof}
  Theorem~\ref{thm:stcohbwweak} and Theorem~\ref{thm:js}. For \ref{enum:bwcomp3} note that the jump of a slowly converging Cauchy sequence computes a fast converging subsequence.
\end{proof}
Theorem~\ref{thm:bw} gives rise to another proof of this theorem and Theorem~\ref{thm:js}: 
Let $d$ be a degree containing solutions to all recursive instances of \lp{BW}. Since \lp{BW} is equivalent to \lp[\Sigma^0_1]{WKL} any degree $d\gg 0'$ suffices. Thus we may assume that $d$ is $low$ over $0'$, i.e. $d'\equiv 0''$. 
Now let $e$ be a degree containing solutions to all recursive instances of \lp{BW_{weak}}. Since the choice of a fast convergent subsequence of a slow convergent subsequence is equivalent to the halting problem, $e$ may be chosen such that $e'\equiv d$. Thus $e''\equiv 0''$ or in other words $e$ is $low_2$. 
 
Theorem~\ref{thm:bwcomp}.\ref{enum:bwcomp1} improves a result obtained by Le~Roux and Ziegler in \cite[section 3]{RZ08}, which only considers full Turing jumps.

\bibliographystyle{amsalpha}
\bibliography{swkl}

\end{document}